\documentclass[12pt]{amsart}
\usepackage{amsmath,amssymb,amsfonts,latexsym,graphicx}
\usepackage{mathrsfs}
\usepackage[dvips]{geometry}
\usepackage{array}
\usepackage{epsf}
\usepackage{epsfig}
\newtheorem*{lemma}{Lemma}
\newtheorem*{prop}{Proposition}
\newtheorem*{thm}{Theorem}
\newtheorem*{cor}{Corollary}
\newcommand{\iso}{\overset{\sim}{\rightarrow}}

\newcommand{\End}{\operatorname{End}}
\newcommand{\Hom}{\operatorname{Hom}}

\newcommand{\ad}{\operatorname{ad}}

\newcommand{\twoheaddownarrow}{\rotatebox[origin=c]{90}{$\twoheadleftarrow$}}
\newcommand{\nc}{\newcommand}
\nc{\Ker}{\operatorname{Ker}} \nc{\rker}{\operatorname{rKer}}
\nc{\im}{\operatorname{Im}}
\nc{\stab}{\operatorname {Stab}}
\nc{\ann}{\operatorname {Ann}}
\nc{\Id}{\operatorname {Id}}

\nc{\Ext}{\operatorname {Ext}}

\begin{document}


\title[Relative Yangians]{Modules for Relative Yangians (Family Algebras) and Kazhdan-Lusztig Polynomials}
\author[Anthony Joseph]{Anthony Joseph}\footnote {Work supported in part by Israel Science Foundation Grant, no. 710724.}

\date{\today}
\maketitle

\vspace{-.9cm}\begin{center}
Donald Frey Professional Chair\\
Department of Mathematics\\
The Weizmann Institute of Science\\
Rehovot, 76100, Israel\\
anthony.joseph@weizmann.ac.il
\end{center}\

%
\date{\today}
\maketitle

\

\medskip


\

\textbf{Abstract}

\

Let $\mathfrak g$ be a complex simple Lie algebra and let$\mod U(\mathfrak g)$ be the category of finite dimensional $U(\mathfrak g)$ modules.  The relative Yangian $Y_V(\mathfrak g)$ with respect to the pair $\mathfrak g,V: V \in\mod U(\mathfrak g)$ is defined to be the $\mathfrak g$ invariant subalgebra of $\End V \otimes U(\mathfrak g)$ with respect to the natural ``diagonal" action.  According to recent work (see \cite [Sect. 5]{KN} and references therein) the finite dimensional simple modules of the Yangians for $\mathfrak g= \mathfrak {sl}(n)$ or the twisted Yangians for $\mathfrak g =\mathfrak {sp}(2n), \mathfrak {so}(n)$ are described by the simple modules of relative Yangians $Y_V(\mathfrak g): V \in\mod U(\mathfrak g)$.

Here a classification of the simple modules of a relative Yangian is obtained simply and briefly as an advanced exercise in Frobenius reciprocity inspired by a Bernstein-Gelfand equivalence of categories \cite {BG}. An unexpected fact is that the dimension of these modules are determined by the Kazhdan-Lusztig polynomials and conversely the latter are described in terms of dimensions of certain extension groups associated to finite dimensional modules of relative Yangians.

\section{Introductory Preliminaries}

The ground field is assumed to be the complex field $\mathbb C$ throughout.  Generally speaking if a set is denote by $X$, then elements of $X$ are denoted by $x,x',x'',\ldots$.  If $M,N$ are $\mathbb C$ vector spaces, then we write simply write $\Hom_\mathbb C(M,N)$ as $\Hom(M,N)$.  If $A$ is an algebra and $M,N$ are $A$ modules, we shall sometimes write $\Hom_A(M,N)$ as $Hom(M,N)$, for example in Lemmas \ref {3.1} and \ref {7.3}.

\subsection{}\label{1.1}

Let $\mathfrak g$ be a finite dimensional Lie algebra and $U(\mathfrak g)$ its enveloping algebra.  The latter is a Hopf algebra and we let $\varepsilon$ (resp. $\sigma$) denote its augmentation (resp. antipode). One has $\sigma(x)=-x$, for all $x \in \mathfrak g$.  It is also known as the principal antiautomorphism of $U(\mathfrak g)$.

Given $u \in U(\mathfrak g)$, we write the image of $u$ under the coproduct on $U(\mathfrak g)$ as $u_1\otimes u_2$, with the understanding that this represents a finite sum.  It is the shorthand version of Sweedler's notation introduced in \cite [1.1.8]{J2}.

Let $V$ be a finite dimensional left $U(\mathfrak g)$ module, so there exists an algebra homomorphism $\Theta: U(\mathfrak g) \rightarrow \End V$. Consider $V^*$ as a left $U(\mathfrak g)$ module through transport of structure and the antipode $\sigma$. Define a left action of $U(\mathfrak g)$ on $V\otimes V^*$ through the above two actions and the coproduct, that is $(u,v\otimes \xi) \mapsto \Theta(u_1)v\otimes \Theta(\sigma(u_2))$.

Recall that there is an isomorphism $\Xi: v\otimes \xi \mapsto (v' \mapsto \xi(v')v)$ of $V\otimes V^*$ onto $\End V$ of $\End V$ bimodules.   Identify $\End V$ with $V\otimes V^*$ through $\Xi$. Then with respect to the natural $U(\mathfrak g)$ bimodule structure of $\End V$, the above left action of $U(\mathfrak g)$ on $V\otimes V^*$ transports to a left action of $U(\mathfrak g)$ on $\End V$ given by $(u,e)\mapsto (\ad u)e:=u_1e\sigma(u_2)$, for all $u \in U(\mathfrak g), e \in \End V$.  It is called the adjoint action and is compatible with multiplication.

Consider $U(\mathfrak g)$ as a left (resp. right) $U(\mathfrak g)$ module through left (resp. right) multiplication.  (Of course these are not compatible with multiplication in $U(\mathfrak g)$.)

Consider $\End V \otimes U(\mathfrak g)$ as a left
$U(\mathfrak g)$ module via the coproduct on $U(\mathfrak g)$, and the above given left actions on the two factors. Consider $\End V \otimes U(\mathfrak g)$ as a right $U(\mathfrak g)$ module via right multiplication on the second factor (and the trivial action on the first factor). This makes $\End V \otimes U(\mathfrak g)$, a $U(\mathfrak g)$ bimodule and hence a $U(\mathfrak g)\otimes U(\mathfrak g)=U(\mathfrak g \times \mathfrak g)$ module via the antipode $\sigma$ on $U(\mathfrak g)$.

Let $\mathfrak k$ denote the diagonal copy $(x,x):x \in \mathfrak g$ of $\mathfrak g$ in $\mathfrak g \times \mathfrak g$. Let $\textbf{G}$ be the connected simply-connected algebraic group with Lie algebra $\mathfrak g$. Let $\textbf{K}$ be the unique closed subgroup of $\textbf{G} \times \textbf{G}$ with Lie algebra $\mathfrak k$.

The Yangian of $\mathfrak g$ relative to $V$ is defined to be the subalgebra
 $Y_V(\mathfrak g)=(\End V \otimes U(\mathfrak g))^\mathfrak k$ of $\End V \otimes U(\mathfrak g)$.

Of course the $Y_V(\mathfrak g)$ are a very natural objects to study. They were called ``family algebras" by Kirillov in \cite {Ki1}.  However in this and subsequent papers \cite {Ki2}, \cite {R} little of their representation theory was developed and certainly not anything resembling our main result (Theorem \ref {1.7}).  In this it is particularly notable that there is no mention of these works in either \cite {KN} or \cite {KNV} which were the source of our present inspiration.

We could have given $\End V \otimes U(\mathfrak g)$ a more ``symmetric" $U(\mathfrak g)$ bimodule structure.  Namely equip $\End V$ with its natural left (resp. right) $U(\mathfrak g)$ module structure as above  and then consider $\End V \otimes U(\mathfrak g)$ as a left (resp. right) $U(\mathfrak g)$ module via the coproduct on $U(\mathfrak g)$.  This is \textit{not} the same as the previous bimodule structure. However the resulting $\mathfrak k$ module structures coincide. (Of course this is well-known; but some more details can be found in \ref {5.0}.)  Thus the invariant subalgebra $Y_V(\mathfrak g)=(\End V \otimes U(\mathfrak g))^\mathfrak k$ is the same in both cases.

In Theorem \ref {1.7} we classify the set of simple  $Y_V(\mathfrak g)$ modules determining also (\ref {6.7}) their dimensions (in terms of Kazhdan-Lusztig polynomials).  As a bonus we show that certain of their extension groups with respect to analogues of Verma modules are determined by and determine these polynomials (Corollary \ref {7.6}).

The above results reinterpret and notably extend to exceptional Lie algebras some results of Khoroshkin and Nazarov (see Remarks and references in \ref {2.1}). Here we mention that the present notion of a relative Yangian is fairly explicit in their work as is also the study of finite dimensional simple modules of Yangians and of twisted Yangians through relative Yangians.  Earlier work on the study of such representations can be found in the book of Molev \cite {M}.

The way the results on simple modules for relative Yangians can be used to describe finite dimensional representations of the Yangians $Y(\mathfrak g)$  for $\mathfrak g= \mathfrak {sl}(n)$ or of the twisted Yangians $Y(\mathfrak g)$ for $\mathfrak g =\mathfrak {sp}(2n), \mathfrak {so}(n)$ is discussed in \cite {KN}.  We summarize this briefly below.  In this $\mathfrak g$ is assumed to be of classical type.

Let $V$ be a locally finite $\mathfrak g$ module. Since $\mathfrak g$ is reductive, $V$ is a direct sum of finite dimensional modules and as a consequence we may for our purposes assume that $V$ is finite dimensional.  Now let $H(V)$ denote either the Clifford algebra $C(V)$ of $V$ or the Weyl algebra $D(V)$ of $V$.  As is well-known $C(V)$ is at most the sum of two endomorphisms rings of finite dimensional $\mathfrak g$ modules (which can be explicitly constructed from $V$). However the stucture of $D(V)$ is quite different and in particular it has no finite dimensional modules.

Similar to the case when $H(V) =\End V$  consider the algebra $H(V)\otimes U(\mathfrak g)$ as a left and a right $U(\mathfrak g)$ module under diagonal action (that is via the coproduct) on both factors and then as a
$\mathfrak k$ module.  For one of the two possible choices of $H(V)$ (determined by a parameter $\theta$ specified in \cite [0.2]{KN}) there exists \cite [Props. 4.1, 4.2]{KN} an Olshanski homomorphism of $Y(\mathfrak g)$ onto  $(H(V)\otimes U(\mathfrak g))^\mathfrak k$.    In this it is \textit{not} asserted that every homomorphism of $Y(\mathfrak g)$ to a finite dimensional matrix ring factors through some Olshanski homomorphism.  Consequently it is not immediate that every finite dimensional module $Y(\mathfrak g)$ obtains as a finite dimensional module for some $(H(V)\otimes U(\mathfrak g))^\mathfrak k$. However the version of Theorem \ref {1.7} given by Khoroshkhin and Nazaraov \cite {KN} (see also \ref {2.1}, Remarks) gave a list of simple finite dimensional modules for $Y(\mathfrak g)$ which the authors showed to coincide with list obtained (by several groups of authors \cite [Chaps. 3,4]{M} and references therein) using highest weights and Drinfeld polynomials. (Here only the case $H(V)=C(V)$ is needed to describe the \textit{finite dimensional} simple modules.)  The reader may consult \cite [Sect. 5]{KN} for the details of this analysis.

 Some of the theory developed here for $H(V)=\End V$ (which immediately applies to the case $H(V)=C(V)$) goes through for $H(V)=D(V)$.  However unlike the conclusion of Theorem \ref {1.7}(i) we cannot expect that every simple $(D(V) \otimes U(\mathfrak g))^\mathfrak k_\lambda$ (with as before $\lambda \in \mathfrak h^*$ a dominant element defining the central character) takes the form $\Hom_{U(\mathfrak g)}(M(\lambda), S(V) \otimes L(\mu))$.  Nor can one expect the remaining parts of Theorem \ref {1.7} to hold. Indeed already the representation theory of $D(V)$ is highly pathological.  At the very least it will be necessary to restrict to a subcategory of  $(D(V) \otimes U(\mathfrak g))^\mathfrak k$ modules.  Besides this, the natural analogue of Lemma \ref {5.0} fails.  Thus whilst $\End V$ (for $V$ finite dimensional) is isomorphic to $V\otimes V^*$ as an $\End V$ bimodule, $D(V)$ is only isomorphic to $S(V)\otimes S(V^*)$ as a vector space but not as a $D(V)$ bimodule using the natural action of $D(V)$ on the two factors.  The former result is an essential part of our analysis and plays a role in several steps, for example in \ref {2.4} and in \ref {5.0}.

 We may conclude from the above discussion that there is still much more to be done especially in the context of studying representations of Yangians.  However the present theory discussing the representation theory of relative Yangians is fairly complete.

\subsection{}\label{1.2}

In the sequel we shall fix $V$ and simply write $A=\End V \otimes U(\mathfrak g), B=Y_V(\mathfrak g)$.

Since $\End V$ is a $U(\mathfrak g)$ module under adjoint action
we may also form the semi-direct (also called smash) product $A^\ltimes:=\End V \ltimes U(\mathfrak g)$. It is $A$ as a vector space but with  multiplication given by $(e\otimes u).(e'\otimes u')=e(\ad u_1)e'\otimes u_2u'= eu_1e'\sigma(u_2)\otimes u_3u'$. One checks that the map $\zeta:e\otimes u \mapsto eu_1\otimes u_2$ extends to an algebra isomorphism of
$A^\ltimes:=\End V \ltimes U(\mathfrak g)$ onto $A:=\End V \otimes
U(\mathfrak g)$, with inverse $e\otimes u \mapsto e\sigma(u_1)\otimes u_2$.

Take $e'\otimes u' \in B$, using also the convention that this may represent a finite sum.  Then with respect to \textit{smash} product multiplication $(e\otimes u).(e'\otimes u')=e(\ad u_1)e' \otimes u_2u'=e(\ad u_1)e'\otimes (\ad u_2)u'u_3=e\varepsilon(u_1)e'\otimes u'u_2=ee'\otimes u'u$, since $(\ad u_1)e'\otimes (\ad u_2)u'=\varepsilon(u_1)(e'\otimes u')$ by the ad-invariance of $e'\otimes u'$. (Of course a little familiarity with Hopf algebra calculations is needed here.  For this the reader may consult \cite [1.1.8]{J2} or indeed any text on Hopf algebras.) Through the principal antiautomorphism
of $U(\mathfrak g)$ it follows that the algebra structure of $B$
is independent of whether its algebra structure is derived from
taking invariants in $A$ or in $A^\ltimes$.

One may note also that the left (resp. right) action of $U(\mathfrak g)$ on $A$ defined in \ref {1.2} is just that given by left (resp. right) multiplication by $\Id_V \otimes U(\mathfrak g)$ with respect to the \textit{smash product}.  This is the first reason for our choice of bimodule structure in \ref {1.1}.


\subsection{}\label{1.3}

If $M$ is a (left) $U(\mathfrak g)$ module then we shall always consider $V\otimes M$ as an $A$ module under, component-wise action $\phi$ and as a $A^\ltimes$ module through the action $\phi^\ltimes:((e\otimes a),(v \otimes m))\mapsto (e\otimes a).(v \otimes m) :=ea_1v\otimes a_2m$.  It is easily verified that the diagram

$$
\begin{array}{ccc}
 A^\ltimes \otimes (V \otimes M)  & \stackrel{\zeta}{\iso}& A \otimes (V\otimes M) \\
 \phi^\ltimes\downarrow &&  \downarrow \phi \\
 V\otimes M& = &  V\otimes M
\end{array}
$$
is commutative.

\subsection{}\label{1.4}

Let $Z(\mathfrak g)$ denote the centre of $U(\mathfrak g)$.  It is clear that $\Id_V \otimes Z(\mathfrak g)$ belongs to the centre $Z(B)$ of $B$.  By the Kostant separation theorem \cite [8.2.4]{D}, it follows that $B$ is a finite free module over $\Id_V \otimes Z(\mathfrak g)$. Then as is well-known its simple modules are finite dimensional.  Indeed suppose that the rank of $B$ is $< 2n$. Then the standard identity $S_{2n}$ is an identity for $B$, the proof being as in \cite [Lemma 6.2.2]{H}. Combining the Jacobson density theorem \cite [Thm. 2.1.2]{H} with a result of Kaplansky \cite [Lemma 6.3.1]{H}, it follows that every simple $B$ module has dimension $\leq n$.

\subsection{}\label{1.5}

Fix a Cartan subalgebra $\mathfrak h$ of $\mathfrak g$ and let $W$ be the Weyl group corresponding to the pair $(\mathfrak g, \mathfrak h)$. Let $\Delta$ denote the set of non-zero roots. Given $\alpha \in \Delta$, let $\alpha^\vee \in \mathfrak h$ be the corresponding coroot and $s_\alpha$ the corresponding reflection.  Let $W$ denote the (Weyl) group that the $s_\alpha:\alpha \in \Delta$ generate. Fix a set $\pi\subset \Delta$ of simple roots and set $\Delta^\pm:=\Delta \cap \pm \mathbb N\pi$. Let $\rho$ denote the half sum of the roots in $\Delta^+$.  Recall that we have a translated action of $W$ on $\mathfrak h^*$ defined by $w.\lambda=w(\lambda +\rho)-\rho$, for all $\lambda \in \mathfrak h$.

Fix $\lambda \in \mathfrak h^*$. Set $\Delta_\lambda: =\{\alpha \in \Delta|\alpha^\vee(\lambda) \in \mathbb Z\}$ and $W_\lambda: =\{ w \in W|w\lambda - \lambda \in \mathbb Z\pi\}$.  As noted by Jantzen  \cite [Satz 1.3]{Ja} $\Delta_\lambda$ is a root system with Weyl group $W_\lambda$.

Call $\lambda \in \mathfrak h^*$ dominant when $\alpha^\vee(\lambda+\rho) \in \mathbb N$, for all $\alpha \in \Delta^+_\lambda$.  Call $\lambda \in \mathfrak h^*$ regular if $\stab_{W.}\mathfrak \lambda:=\{w \in W|w.\lambda=\lambda\}$ is reduced to the identity. Let $P$ (resp. $P^+$) denote the set of weights (resp. dominant weights).

Set $\Lambda=\lambda +P$.   It is clear that $W_\mu=W_\lambda, \Delta_\mu=\Delta_\lambda$, for all $\mu \in \Lambda$.

We say that $\lambda$ is in general position when $W_\lambda$ is reduced to the identity.

Let $M(\lambda)$ denote the Verma module with highest weight $\lambda$.  Let $L(\lambda)$ denote its unique simple quotient.

\subsection{}\label{1.6}

Let $S$ be a simple (and hence finite dimensional) $B$ module. By Schur's lemma, $Z(B)$ acts by scalars on $S$ and hence by \ref {1.4} defines a maximal ideal of $Z(\mathfrak g)$.  As noted in \cite [7.4.7]{D}, any such maximal ideal is given as $\ann_{Z(\mathfrak g)}M(\lambda)$, for some unique dominant element $\lambda \in \mathfrak h^*$.

\

In the sequel we shall fix $\lambda \in \mathfrak h^*$ dominant and consider only $B$ modules annihilated by $Id_V \otimes \ann_{Z(\mathfrak g)}M(\lambda)$.

Recall \cite [8.4.3]{D} that $\ann_{U(\mathfrak g)} M(\lambda)=U(\mathfrak g)\ann _{Z(\mathfrak g)}M(\lambda)$.  We further set $U_\lambda:=U(\mathfrak g)/\ann_{U(\mathfrak g)} M(\lambda)$ and $A_\lambda:=\End V \otimes  U_\lambda, B_\lambda:= A_\lambda^\mathfrak k$.

The $U(\mathfrak g)$ bimodule structure of $A$ descends to $A_\lambda$.  Through its definition $A_\lambda$ is annihilated by the right action of $\Id_V \otimes \ann_{Z(\mathfrak g)}M(\lambda)$ but \textit{not} by its left action.  On the other hand $B_\lambda$ is annihilated by both the left and the right action of $\Id_V \otimes \ann_{Z(\mathfrak g)}M(\lambda)$.

Clearly a $B$ module $M$ is a $B_\lambda$ module if and only if it is annihilated $\Id_V \otimes \ann_{Z(\mathfrak g)}M(\lambda)$.

\subsection{}\label{1.7}

We shall prove the following theorem which describes all simple $B_\lambda$ modules.  In this we consider $V \otimes L(\mu):\mu \in \Lambda$ as a $U(\mathfrak g)$ module through the coproduct on $U(\mathfrak g)$.

Fix $\lambda \in \mathfrak h^*$ dominant.  Then $\Hom (M(\lambda), V\otimes L(\mu))$ is an $A$ module, hence an $A^\ltimes$ module, through the action on the second factor. It is also a right $U_\lambda$ module by the action on the first factor. Thus $S(\mu):=\Hom_{U(\mathfrak g)}(M(\lambda), V\otimes L(\mu))$, is a $B_\lambda$ module by restriction.

\begin {thm}

\

(i)  Every simple $B_\lambda$ module takes the form $S(\mu)$, for some  $\mu \in \Lambda$.

\

(ii) If $S(\mu) \neq 0$, then it is a simple $B_\lambda$ module.

\

(iii) If two simple $B_\lambda$ modules $S(\mu),S(\mu^\prime)$ are isomorphic, then $\mu = \mu^\prime$.

\end {thm}

\textbf{Remark.}  The possible non-vanishing of $S(\mu)$ is discussed below and in Section 7.

\subsection{}\label{1.8}

The proof of the above theorem is basically an advanced exercise in Frobenius reciprocity. However there are some delicate points, particularly with respect to actions and because of this our analysis may seem to be a little over-pedantic.  In the Section $2$ we introduce the main characters and note (Proposition \ref {2.6}) a fairly standard equivalence of categories.  In Section $3$ we describe the Bernstein-Gelfand \cite {BG} equivalence of categories.  This had been developed to give a more algebraic proof of results of Zhelobenko (the latter having been carefully worked over by Duflo \cite {Du}).  Yet the Bernstein-Gelfand analysis has the special advantage of using projective covers of simple highest weight modules.  These are better behaved than Verma modules as they occur as direct summands, a fact particularly used in Section $5$. Moreover the magic of projective covers is used several times in the text, for example in the proof of Proposition \ref {3.5} and Lemma \ref {7.5}.

Thus in Sections $2,3$ we have given enough details of proofs so that the reader may easily reconstruct the full details alone. Some may find this useful, others superfluous. In any case similar techniques are used in Section 7.  For our own convenience we follow partly the treatment in \cite [8.4]{J2}, which shows, for what it is worth, that the analysis goes over to the quantum case.  Section $4$ establishes a key result relating induction to ``restriction".  The main new ideas are contained in Section $5$ giving a proof of the main theorem.  In Section $6$ we give a necessary condition for $S(\mu)$ to be non-zero and more generally calculate its dimension in terms of values of Kazhdan-Lusztig polynomials.
As a bonus we also compute the projective cover $Q(\mu)$ of the simple $B_\lambda$ module $S(\mu)$ and show that $B_\lambda$ is a direct sum of the $Q(\mu)$ with multiplicity $\dim S(\mu)$.

In Section 7, we summarize our results through a further equivalence of categories.  This shows in particular that the Kazhdan-Lusztig polynomials can be expressed in terms of dimensions of extension groups involving modules for the relative Yangians $B_\lambda$ which we recall are finite dimensional algebras and can in principle be described explicitly in terms of generators and relations. One may remark that when $\lambda$ is in general position, $B_\lambda$ is a finite direct sum of full matrix algebras (Lemma \ref {6.6}).

\

\textbf{Acknowledgements.}  I am grateful to Maxim Nazarov for bringing my attention to this problem when he visited the Weizmann Institute in March 2010 and in March 2012 and for explaining to me the results in \cite {KN}. I would also like to thank Jacques Alev for the invitation to speak in Reims during the first week of December 2012, in which these results were presented.  The gift from the author of \cite {M} saved me much trouble.  Finally I would like to thank L. W. Small for references \cite {ARS}, \cite {V}, J.T. Stafford for the proof of the lemma in Remark {6.5} and I. Reiten for her comments on Remark \ref {6.5}.

\section{The Bernstein-Gelfand-Gelfand and Harish-Chandra Categories}

\subsection{}\label{2.1}

Let $M$ be a $U(\mathfrak g)$ module.  We say that $M$ is a weight module if $U(\mathfrak h)$ acts locally finitely and reductively.  Then $M$ is a direct sum of its $\mathfrak h$ weight subspaces $M_\lambda: \lambda \in \mathfrak h^*$.  We say that $M$ is admissible if these subspaces are finite dimensional.

Let $\mathfrak b$ be the Borel subalgebra of $\mathfrak g$ defined by the choice of $\mathfrak h$ and of $\pi$.

The Bernstein-Gelfand-Gelfand $\mathscr O$ category introduced in \cite {BGG} consists of all finitely generated $U(\mathfrak g)$ admissible weight modules on which $U(\mathfrak b)$ acts locally finitely.  It is an abelian category in which all modules have finite length.  The $\mathscr O$ category is stable under tensoring by a finite dimensional $U(\mathfrak g)$ module $V'$.  Moreover the functor $V'\otimes \cdot: M \mapsto V'\otimes M$ on $\mathscr O$ is covariant and exact, with adjoint functor $V'^*\otimes \cdot$. From this one easily checks that it takes projectives to projectives (as is well-known).

Take $\lambda \in \mathfrak h^*$ and let $\mathbb C_\lambda$ be the one-dimensional $U(\mathfrak b)$ module of weight $\lambda$.  Recall that $M(\lambda):=U(\mathfrak g)\otimes _{U(\mathfrak b)}\mathbb C_\lambda$ is the Verma module of highest weight $\lambda$.   It belongs to the $\mathscr O$ category and is projective in $\mathscr O$ if and only if $\lambda$ is dominant.  In this latter case it is the projective cover of $L(\lambda)$.  In general the projective cover $P(\mu)$ of $L(\mu)$ can be obtained by taking a suitable direct summand of the tensor product of some finite dimensional $U(\mathfrak g)$ module with $M(\lambda)$, for $\lambda \in \mu +P$ dominant and regular.  This last fact is noted in \cite {BG} and extended significantly there.

\

\textbf{Remarks.}   Take $M \in \mathscr O$.  Let $\mathfrak g = \mathfrak n^- \oplus \mathfrak h \oplus \mathfrak n^+$, be the triangular decomposition of $\mathfrak g$ associated to the choices made in \ref {1.5}.  Then both $H^0(\mathfrak n^+,M)$ and $H_0(\mathfrak n^-,M)$ are admissible weight modules.   Moreover they are isomorphic if $M$ admits a non-degenerate contravariant form (via transport of structure).  In particular this is the case if $M=V\otimes L(\mu)$.  Again $\Hom_{U(\mathfrak g)}(M(\lambda),M)$ is canonically isomorphic to $H^0(\mathfrak n^+,M)_\lambda$.  Thus $S(\mu)$ as defined in \ref {1.7} is isomorphic to $H_0(\mathfrak n^-,V\otimes L(\mu))_\lambda$.  This allows one to compare Theorem \ref {1.7} with the results of Khoroshkin and Nazarov (\cite {KN}, and references therein).  They have obtained  (i) and (ii) for $\mathfrak g$ classical; but their proofs are far longer, needing to appeal to the explicit description of Yangians as well as results of Drinfeld.  They also obtain (iii) in general; but their proofs are again more complicated using in particular a surjectivity result for the generalized Harish-Chandra map \cite {KNV} which we do not need.  Since we need nothing from the theory of Yangians we shall drop all further reference to them treating the present problem ``for its own sake".

\subsection{}\label{2.2}

Let $H$ be a $U(\mathfrak g)$ bimodule, equivalently a $U(\mathfrak g \times \mathfrak g)$ module.  We say that $H$ is a $(\mathfrak g \times \mathfrak g,\textbf{K})$ module if $U(\mathfrak k)$ acts locally finitely (and hence reductively). (Note that the ``maximal compact" subgroup, viewed as a real form of $\textbf{K}$, of the ``complex group" $\textbf{G}$ is connected, so we are just following the long-established practice of ``Unitarians".) As is well-known such a module $H$ is a direct sum of its $\mathfrak k$ isotypical components.  We say that $H$ is admissible if these occur with finite multiplicity.  The Harish-Chandra category $\mathscr H$ consists of all admissible $(\mathfrak g \times \mathfrak g,\textbf{K})$ modules for which both the left and right action of $Z(\mathfrak g)$ is finite.  It is an abelian category stable under tensoring with finite dimensional modules on the left and on the right.

Take $\lambda \in \mathfrak h^*$ dominant.  The subcategory $\mathscr H_\lambda$ of $\mathscr H$ consists of all modules whose right annihilator contains $\ann_{Z(\mathfrak g)}M(\lambda)$. It is an abelian category stable under tensoring with finite dimensional modules on the left, which is an exact covariant functor on $\mathscr H_\lambda$ taking projectives to projectives.

Let $M,N$ be left $U(\mathfrak g)$ modules.  Then $\Hom(M,N)$ is a $U(\mathfrak k)$ module and we set $F(M,N):=\{a \in \Hom(M,N)| \dim U(\mathfrak k)a < \infty\}$, which is again a $U(\mathfrak g)$ bimodule and hence a $(\mathfrak g \times \mathfrak g,\textbf{K})$ module.

Suppose further that $M,N$ belong to $\mathscr O$.  Then $F(M,N)$ is an admissible $(\mathfrak g \times \mathfrak g,\textbf{K})$ module, as is well-known.  (Indeed since objects in $\mathscr O$ have finite length, one can assume $M,N$ simple, so of highest weight and in this case multiplicity estimates are easy to make.)  Moreover the action of $Z(\mathfrak g)$ on $M$ and on $N$ is finite and so $F(M,N)$ belongs to $\mathscr H$.  Finally if we take $M=M(\lambda)$ with $\lambda$ dominant, then
this module further belongs to $\mathscr H_\lambda$.

Notice that the manner we have defined the $U(\mathfrak g)$ bimodule structure of $A$, gives $A_\lambda = \End V \otimes U_\lambda$, the structure of a $A^\ltimes-U_\lambda$ bimodule.

\subsection{}\label{2.3}

The following result is standard but we shall give a proof for completeness.  Adopt the conventions of \ref {1.3}.

\begin {lemma}

\

(i)  Every $A$ module takes the form $N=V \otimes M$, for some $U(\mathfrak g)$ module $M$.

 \

 (ii) $M$ is simple $\Leftrightarrow$ $N$ is simple.

\end {lemma}

\begin {proof}  Take $n \in N$ non-zero. Since $(\End V)n$ is finite dimensional, it admits a simple submodule which is necessarily isomorphic to $V$ and which we again denote by $V$.  Set $M=\Hom_{\End V}(V,N)$.  It has the structure of a $U(\mathfrak g)$ module, via the action of $U(\mathfrak g)$ on $N$ which commutes with the action of $\End V$.  Frobenius reciprocity is the isomorphism
$$T: \Hom_A(V \otimes M,N) \iso \Hom_{U(\mathfrak g)}(M,\Hom_{\End V}(V,N)),\eqno {(*)}$$
given by $T_\psi(m)(v):=\psi(v\otimes m)$, for all $m \in M, v \in V, \psi \in  \Hom_{\End V \otimes U(\mathfrak g)}(V \otimes M,N)$, with inverse given by $S_\theta(v\otimes m)= \theta(m)(v)$, for all $\theta \in \Hom_{U(\mathfrak g)}(M,\Hom_{\End V}(V,N))$.

Let us show that $S_{\Id M}$ is an isomorphism of $V \otimes M$ onto $N$. Take $n \in N$.  Then $(\End V)n$ is a finite direct sum of copies of $V$.  Each direct summand may be expressed as some $m(V): m \in M$.  Thus we may write $n \in (\End V)n$ in the form of a finite sum $\sum m_i(v_i): m_i \in M, v_i \in V$, consequently $S_{\Id M}$ is surjective.  Suppose $\ker S_{\Id M} \neq 0$.  By the Jacobson density theorem \cite [Thm. 2.1.2]{H} (or by direct matrix computation) there exists a non-zero element of $\ker S_{\Id M}$ in the form $v\otimes m$.  Then via the action of $\End V$, we conclude that $V \otimes m\subset \ker S_{\Id M}$, which translates to the contradiction $m(V)=0$.  This proves (i).

The simplicity of $M$ implies via the Jacobson density theorem the simplicity of $V\otimes M$, whilst the converse is obvious. Hence (ii).

\end {proof}

\subsection{}\label{2.4}

Retain the conventions of \ref {1.3}.

\begin {cor}

\

(i)  Every $A^\ltimes$ module takes the form $N=V \otimes M$, for some $U(\mathfrak g)$ module $M$.

 \

 (ii) $M$ is simple $\Leftrightarrow$ $N$ is simple.

 \

 With respect to the diagonal action of $U(\mathfrak g)$ on $V \otimes M$,

 \

 (iii) $M \in \mathscr O  \Leftrightarrow N \in \mathscr O$.

 \

 If in addition $M$ is a right $U(\mathfrak g)$ module, then

 \

 (iv)  $M \in \mathscr H_\lambda  \Leftrightarrow N \in \mathscr H_\lambda$.

\end {cor}

\begin {proof}
(i) and (ii)  are immediate from the above lemma using the commutative diagram in \ref {1.3}.  One may also give a direct proof in a similar fashion to the proof of the lemma based on the following observations.

 Recall that $N$ is now viewed as an $A^\ltimes$ module. One checks $(x.S)(v)=xS(v)-Sxv$, for all $S \in \Hom_{\End V}(V,N))$, defines an action of $U(\mathfrak g)$ on $\Hom_{\End V}(V,N)$. Then with the same formula for $T$ as in $(*)$ one further checks that

$$T: \Hom_{A^\ltimes}(V \otimes M,N) \iso \Hom_{U(\mathfrak g)}(M,\Hom_{\End V}(V,N)),\eqno {(**)}$$
but now with the action of $U(\mathfrak g)$ on $\Hom_{\End V}(V,N)$ defined above.

Consider (iii).  The assertion $\Rightarrow$ is clear.  If $N$ belongs to $\mathscr O$, then so does $V^*\otimes N=V^*\otimes V\otimes M$ in which $M$ lies as a direct summand. Hence (iii).  The proof of (iv) is similar.

\end {proof}

\subsection{}\label{2.5}

Recall \ref {2.1} and let $\mathscr O^V$ denote the category of $U(\mathfrak g)$ modules of the form $V\otimes M$ with $M \in \mathscr O$, given the diagonal action of $U(\mathfrak g)$. By the above corollary this is just the category of $A^\ltimes$ modules which as $U(\mathfrak g)$ modules belong to $\mathscr O$.  With respect to just $U(\mathfrak g)$ action, $\mathscr O^V$ is a subcategory $\mathscr O$.

Take $\lambda \in \mathfrak h^*$ dominant.  Recall \ref {2.2} and let $\mathscr H^V_\lambda$ denote the subcategory of $\mathscr H_\lambda$ of modules of the form $V\otimes H: H \in \mathscr H_\lambda$.  By the above corollary this is just the subcategory of $\mathscr H_\lambda$ modules which are also left $A^\ltimes$ modules.

\subsection{}\label{2.6}

The above results may be expressed more formally as follows.  As noted in \ref {2.4} we obtain a covariant functor $\mathscr T' = \Hom_{\End V}(V,\cdot)$ from the category of $A^\ltimes$ modules to the category of $U(\mathfrak g)$ modules and a covariant functor $\mathscr T = V\otimes \cdot$ from the category of $U(\mathfrak g)$ modules to the category of $A^\ltimes$ modules, which by $(**)$ is adjoint to $\mathscr T'$.  Moreover let $N$ be an $A^\ltimes$ module. Then the proof of Lemma \ref {2.3} shows that the identity map on $\mathscr T'N$ induces an isomorphism of $\mathscr T \mathscr T' N$ onto $N$.  Moreover by Corollary \ref {2.3}(ii) simple modules are translated to simple modules.

We conclude that $\mathscr T$ with inverse $\mathscr T'$ is an equivalence of categories.  More specifically in view of (iii) and (iv) of Corollary \ref {2.4} we obtain the following

\begin {prop} The functor $\mathscr T$ restricts to an equivalence of categories of $\mathscr O$ (resp. $\mathscr H_\lambda$) onto $\mathscr O^V$ (resp. $\mathscr H^V_\lambda$) with inverse functor $\mathscr T'$.
\end {prop}

\section{The Bernstein-Gelfand Equivalence of Categories}

In this section we extend very slightly results of Bernstein and Gelfand \cite {BG}. This extension can be obtained by applying Proposition \ref {2.6} to the results in \cite {BG}. Most of the stated results will be used in the subsequent sections. Our proofs may be skipped by the reader who prefers to follow \cite {BG}.

\subsection{}\label{3.1}

Recall \ref {2.2} and let $M,N,V'$ be $U(\mathfrak g)$ modules.  One checks that the canonical injection
$$\Hom(\mathbb C,V') \otimes \Hom(M,N) \hookrightarrow \Hom(M,V'\otimes N),$$
is an isomorphism of $A^\ltimes-U(\mathfrak g)$ bimodules.

 In this we take $V'$ to be the fixed finite dimensional $U(\mathfrak g)$ module $V$ used \ref {1.2} in defining $A$.  Then one may check (cf \cite [3.2]{J0}) that the above map restricts to an isomorphism
$$V\otimes F(M,N) \iso F(M,V\otimes N), \eqno{(*)}$$
of $A^\ltimes-U(\mathfrak g)$ bimodules, in which we have identified $\Hom(\mathbb C,V)$ with $V$.  (For $(*)$ to hold it is enough that $V$ be a direct sum of finite dimensional $U(\mathfrak g)$ modules.)

Fix $\lambda \in \mathfrak h^*$ dominant.  Given $N \in \mathscr O^V$. Then (cf. \ref {2.2}) it follows easily that $F(M(\lambda),N) \in \mathscr H^V_\lambda$. Thus we obtain a covariant functor $\mathscr F:=F(M(\lambda),\cdot)$ taking $\mathscr O^V$ to $\mathscr H^V_\lambda$.  By $(*)$, it commutes with tensoring by finite dimensional $U(\mathfrak g)$ modules on the left.

Conversely take $H \in \mathscr H^V_\lambda$.  Then $H\otimes _{U(\mathfrak g)}M(\lambda)$ is a left $A^\ltimes$ module.  Since $H\in \mathscr H_\lambda$ one has $H\otimes _{U(\mathfrak g)}M(\lambda)\in \mathscr O$, as is well-known (see also below).  Thus $\mathscr F^\prime:= \cdot \otimes _{U(\mathfrak g)}M(\lambda)$ is a covariant functor taking $\mathscr H^V_\lambda$ to $\mathscr O^V$.  It clearly commutes with tensoring by finite dimensional $U(\mathfrak g)$ modules on the left.  Because it is defined by tensor product, the functor $\mathscr F'$ is right exact; but can fail to be exact \cite [5.9]{BG}.

\begin {lemma}

\

(i) $\mathscr F$ is an exact functor.

\

(ii)  For all $H \in \mathscr H^V_\lambda, N \in \mathscr O^V$ one has an isomorphism
$$Hom(H,\mathscr FN) \stackrel{\Theta}{\iso} Hom(\mathscr F^\prime H,N),$$
where $Hom$ means homomorphism in the category.

\end {lemma}

\begin {proof}  When $V$ is the trivial $U(\mathfrak g)$ module this result is due to Bernstein and Gelfand \cite {BG}, see also \cite [Thm. 1.16]{GJ}.  The general case is practically immediate from this special case using Corollary \ref {2.4}.

In a little more detail:

(i) is just a consequence of $\lambda$ being dominant \cite [8.4.8]{J2}.

For (ii) observe that Frobenius reciprocity is the isomorphism of vector spaces
$$F:\Hom_{A^\ltimes-U(\mathfrak g)}(H,\Hom(M(\lambda),N)\iso \Hom_{A^\ltimes}(H\otimes_{U(\mathfrak g)}M(\lambda),N),$$
 given by $F_\psi(h\otimes m)=\psi(h)(m)$.  Finally since $H$ is locally finite under $\mathfrak k$ action we can replace $\Hom(M(\lambda),N)$ in the above by its $\mathfrak k$ locally finite part namely $F(M(\lambda),N)$.

\end {proof}

\subsection{}\label{3.2}

By \ref {3.1}(ii) we obtain
isomorphisms
$$ \Theta: Hom(H,\mathscr F \mathscr F'H) \iso Hom(\mathscr F^\prime H,\mathscr F'H), \quad \Theta':Hom(\mathscr FN,\mathscr FN) \iso Hom(\mathscr F^\prime \mathscr FN,N).$$

For all $H \in \mathscr H^V_\lambda$ such that $\mathscr F'H \neq 0$, let $\theta^-_H$  be the map $H \rightarrow \mathscr F \mathscr F'H$ obtained as the inverse image of the identity map on $\mathscr F'H$, that is $\theta^-_H:=\Theta^{-1}(\Id_{\mathscr F'H})$. One may check that $\theta^-_H(h)(m)=h\otimes m$, for all $h \in H, m \in M(\lambda)$, in particular it is a non-zero map.

For all $N\in \mathscr O^V$ such that $\mathscr FN \neq 0$, let $\theta^+_N$ be the map $\mathscr F' \mathscr FN \rightarrow N$ obtained as the image of the identity map on $\mathscr FN$, that is $\theta^+_N:=\Theta'(\Id_{\mathscr FN})$. One may check that $\theta^+_N(h\otimes m)=h(m), \forall h \in \mathscr FN, m \in M(\lambda)$. In particular it is a non-zero map.

\begin {lemma}  For all $H \in \mathscr H^V_\lambda$ finitely generated, $\theta^-_H$ is an isomorphism of $A^\ltimes - U(\mathfrak g)$ bimodules.

\end {lemma}

\begin {proof}

When $V$ is the trivial $U(\mathfrak g)$ module this result is due to Bernstein and Gelfand \cite {BG}. It uses the crucial fact originating in work of Kostant and Duflo (see for example \cite [8.2.4, 8.4.3]{D}) that $\mathscr FM(\lambda)=F(M(\lambda),M(\lambda))$ coincides with its subalgebra $U_\lambda=U(\mathfrak g)/\ann _{U(\mathfrak g)}M(\lambda)$.   A second crucial fact is that a finitely generated module $H \in \mathscr H_\lambda$ can be expressed as a quotient of some $V'\otimes U_\lambda$, for some finite dimensional $U(\mathfrak g)$ module $V'$.  This follows rather easily from the definition of $\mathscr H_\lambda$.

The proof of the general case is similar.  In a little more detail,
take $H=F(M(\lambda),V\otimes M(\lambda))=V\otimes U_\lambda$.  By the second equality $\mathscr F'H=V\otimes M(\lambda)$, whilst $\mathscr F (V\otimes M(\lambda))=H$.  Consequently $\theta^-_H$ is an isomorphism of $A^\ltimes - U(\mathfrak g)$ bimodules when $H=V\otimes U_\lambda$.

On the other hand $\mathscr F$ and $\mathscr F'$ commute with tensoring by finite dimensional $U(\mathfrak g)$ on the left.  It follows by a standard computation, say as in \cite [8.4.4]{J2}, that $\theta^-_H$ is an $A^\ltimes - U(\mathfrak g)$ bimodule isomorphism for all $H \in \mathscr H^V_\lambda$.

\end {proof}

\subsection{}\label{3.3}

\begin {lemma}  Take $L \in \mathscr O^V$ simple.  Then either $\mathscr FL=0$ or $\mathscr FL$ is simple.  Moreover every simple object in $\mathscr H^V_\lambda$ is so obtained.
\end {lemma}

\begin {proof} When $V$ is the trivial $U(\mathfrak g)$ module this result is due to Bernstein and Gelfand \cite {BG}.   The last part is an older result of Zhelobenko for which a careful proof was given by Duflo \cite [Section I, 4.5]{Du}, using \cite [1.12]{GJ} to relate the presentation in this last reference with that given here.

The general case then follows easily from Corollary \ref {2.4}.  Alternatively one may proceed as follows.

Suppose $\mathscr F L \neq 0$ and let $\iota:H\hookrightarrow \mathscr F L$ be a proper embedding.  Then $\Theta(\iota):\mathscr F'H \rightarrow L$ is non-zero, hence surjective.  By exactness of $\mathscr F$ and Lemma \ref {3.2}, the composed map $H \stackrel{\theta^-_H}{\iso}\mathscr F \mathscr F' H\stackrel{\mathscr F(\Theta(\iota))}{\twoheadrightarrow}\mathscr F L$, is surjective.  One checks that it is given by $h \mapsto (m \mapsto \iota(h)m))$, which is just the original embedding $\iota$. This gives the required contradiction.

\end {proof}

\subsection{}\label{3.4}

A consequence of Lemma \ref {3.3} is that every module in $\mathscr H^V_\lambda$ has finite length.  This is an immediate consequence of the result for $V$ being the trivial module and also has the same proof (see \cite [8.4.7]{J2} for example).

Let $\mathscr O^V_\Lambda$ be the subcategory of $\mathscr O^V$ modules with weights in $\Lambda=\lambda+P$.  It is clear that if $H \in \mathscr H^V_\lambda$, then $\mathscr F'H \in \mathscr O^V_\Lambda$.

We thus obtain the following improvement to Lemma \ref {3.2}.

\begin {cor}  For all $H \in \mathscr H^V_\lambda$, $\theta^-_H$ is an isomorphism of $A^\ltimes - U(\mathfrak g)$ bimodules.  In particular given $H \in \mathscr H^V_\lambda$, there exists $M \in  \mathscr O^V_\Lambda$ such that $H=F(M(\lambda),M)$.

\end {cor}

\textbf{Remark.} Notice this also means that the multiplicity of a simple $\mathfrak k$ module in a Harish-Chandra module $H \in \mathscr H_\lambda$ (so in $\mathscr H^V_\lambda$) is finite.  Of course this is well-known.

\subsection{}\label{3.4.1}

Recall that we fixed a dominant element $\lambda \in \mathfrak h^*$.

Take $L \in \mathscr O$ simple and let $P$ denote its projective cover in $\mathscr O$.  The following is elementary but we indicate the proof anyway as this observation is rather crucial in Section 5.

\begin {lemma} The following two assertions are equivalent.

\

(i) $\mathscr FL \neq 0$,

\

(ii) There exists a simple finite dimensional $U(\mathfrak g)$ module $S$ such that $P$ is a direct summand of $S^*\otimes M(\lambda)$.

\end {lemma}

\begin {proof} Assume (ii) holds.   Then we have $U(\mathfrak g)$ module surjections $S^* \otimes M(\lambda)\twoheadrightarrow P {\twoheadrightarrow} L$. Applying Frobenius reciprocity gives $\Hom_{U(\mathfrak g)}(S^*,\Hom(M(\lambda),L))\neq 0$. A fortiori (i) holds. Conversely suppose $\mathscr FL \neq 0$ and let $S^*$ be a simple submodule of $\mathscr FL=F(M(\lambda),L)$.  Then by Frobenius reciprocity there is a surjective map $S^* \otimes M(\lambda)\twoheadrightarrow L$.  Yet $S^*\otimes M(\lambda)$ is projective hence a direct sum of projective indecomposables in which $P$ being the projective cover of $L$ must appear.
\end {proof}

\subsection{}\label{3.4.2}

In the conventions and notation of \ref {3.4.1}, let $P$ be an indecomposable direct summand of $S^* \otimes M(\lambda)$.

\begin {lemma}  $\mathscr FP$ is projective in $\mathscr H_\lambda$.
\end {lemma}

\begin {proof} Through tensor product with finite dimensional modules and direct sum decomposition, it is enough to show that $\mathscr FM(\lambda)=U_\lambda$ is projective in $\mathscr H_\lambda$.

Suppose $H\twoheadrightarrow H'$ is a surjective map of objects in $\mathscr H_\lambda$.  Since $\mathfrak k$ acts locally finitely and is semisimple, hence acts reductively, we obtain a surjection $H^\mathfrak k \twoheadrightarrow H^{'\mathfrak k}$ of invariants.  View $H^\mathfrak k$ as a trivial $\mathfrak g$ module.  Then $U(\mathfrak g)H^\mathfrak k= H^\mathfrak k U(\mathfrak g)$ is a $U(\mathfrak g)$ bimodule quotient of $H^\mathfrak k \otimes U_\lambda$, which is just a finite number, namely $\dim H^\mathfrak k$, of copies of $U_\lambda$.  Of course a similar assertion holds for $H'$.

On the other hand an element $\psi' \in \Hom_{U(\mathfrak g)-U(\mathfrak g)}(U_\lambda,H')$ is completely determined by $h':=\psi'(1) \in H^{'\mathfrak k}$, which by the first paragraph above generates a quotient $U_\lambda/I'$ of $U_\lambda$, so then $\psi'$ is the canonical projection $U_\lambda \twoheadrightarrow U_\lambda/I'$.  The inverse image of $h'$ in $H^\mathfrak k$ generates a quotient $U_\lambda/I$ of $U_\lambda$. which surjects to $U_\lambda/I'$.  Then the canonical projection $\psi:U_\lambda \twoheadrightarrow U_\lambda/I$ is the required inverse image of $\psi'$ to establish the projectivity of $U_\lambda$.

\end {proof}

\subsection{}\label{3.5}

Take $H \in \mathscr H_\lambda$ simple. By Lemma \ref {3.3} we may write $H=\mathscr FL$, for some $L \in \mathscr O$ simple.  Then the non-zero map $\theta^+_L:\mathscr F' H \rightarrow L$ is surjective.
In particular $\mathscr F'H \neq 0$.

Let $V$ be a finite dimensional $U(\mathfrak g)$ module.  Recall that $\lambda \in \mathfrak h^*$ is assumed dominant, so then $M(\lambda)$ is projective in $\mathscr O$. Let $\{P_i\}_{i\in I}$ be the set of indecomposable direct summands of $V^* \otimes M(\lambda)$.  Then $P_i:i \in I$ is an indecomposable projective module, hence the projective cover of some simple highest weight module $L_i$.  By Lemma \ref {3.4.1} the simples in $\mathscr O$ which so arise are exactly those for which $\mathscr FL\neq 0$.  We remark, though we do not use these facts, that this can be at most the simples in $\mathscr O$ whose weights lie in $\lambda +P$ and only exactly those if $\lambda$ is regular (see proof of \ref {3.7}).

The following result can be read off from \cite [Sect. 5.4]{BG}.  We give the proof for completeness.

\begin {prop}

\

(i) $\mathscr FL_i\neq 0$.  In particular $\theta^+_{L_i}:\mathscr F' \mathscr F L_i \rightarrow L_i$ is surjective.

\

(ii) $\theta^+_{P_i}$ is an isomorphism of $\mathscr F' \mathscr F P_i$ onto $P_i$.

\

(iii)  $\mathscr FL_i$ is the unique simple $U(\mathfrak g)$ bimodule quotient of $\mathscr FP_i$.  In particular $\mathscr FP_i$ is the projective cover of $\mathscr FL_i$.

\

(iv) Take $L,L' \in \mathscr O$ simple.  Assume $\mathscr FL$ and $\mathscr FL'$ are non-zero and isomorphic.  Then $L$ and $L'$ are isomorphic.
\end {prop}

\begin {proof} (i) follows from Lemma \ref {3.4.1}.

  Recall that there is a surjective map $\phi:P_i\rightarrow L_i$.  Then through the covariant functors $\mathscr F'$ and  $\mathscr F$ which are both right exact, we obtain a surjective map $\Phi:=\mathscr F' \mathscr F \phi:\mathscr F' \mathscr F P_i \rightarrow \mathscr F' \mathscr F L_i$. Functoriality (specifically Lemma \ref {3.1}(ii)) gives a commutative diagram

$$
\begin{array}{ccc}
  \mathscr F' \mathscr F P_i & \stackrel{\theta^+_{P_i}}{\longrightarrow}& P_i\\
 \Phi\twoheaddownarrow && \downarrow  \phi \\
 \mathscr F' \mathscr F L_i& \stackrel{\theta^+_{L_i}}{\twoheadrightarrow} &  L_i
\end{array}
$$

 Yet $L_i$ is the unique quotient of $P_i$, so commutativity forces $\theta^+_{P_i}$ to be surjective.

One checks easily that $\theta_{M(\lambda)}$ is an isomorphism of $\mathscr F' \mathscr F M(\lambda)$ onto $M(\lambda)$.  Since these functors commute with tensor product, $\theta^+_{(V^*\otimes M(\lambda))}$ is an isomorphism of $\mathscr F' \mathscr F (V^* \otimes M(\lambda)$ onto $V^*\otimes M(\lambda)$. Then direct sum decomposition and comparison of weight space multiplicities (which are finite) forces $\theta^+_{P_i}$ to be also injective.  Hence (ii).

Let $K_i$ be a simple quotient of $\mathscr F P_i$. Since $K_i \in \mathscr H_\lambda$, we obtain $\mathscr F'K_i \neq 0$, by the first paragraph in \ref {3.5}. The covariant functor $\mathscr F'$ is right exact, so using the inverse of the isomorphism established in (ii) we obtain a composed map $P_i \iso \mathscr F' \mathscr F P_i \twoheadrightarrow \mathscr F'K_i$, which is surjective. Thus the non-zero module $\mathscr F'K_i$ has a non-zero simple quotient which is necessarily the unique simple quotient $L_i$ of $P_i$.   Since $\mathscr F$ is right exact, we obtain using Lemma \ref {3.2} that $K_i \iso \mathscr F \mathscr F' K_i \twoheadrightarrow\mathscr F L_i$.  Yet $\mathscr FL_i \neq 0$ by (i), hence simple by Lemma \ref {3.3}.  We conclude that $\mathscr FL_i$ is the unique simple quotient of $\mathscr FP_i$.  Yet $\mathscr FP_i$ is projective by Lemma \ref {3.4.2}.  Hence (iii).

The isomorphism $\mathscr FL \iso \mathscr FL'$ of simple modules lifts (as always) to an isomorphism of their projective covers $\mathscr FP \iso \mathscr FP'$. The assumption that $\mathscr FL$ is non-zero implies by Lemma \ref {3.4.1} that its projective cover $P$ is some $P_i$ as defined above and occurring in (ii).  The same holds for $L'$.  Applying the functor $\mathscr F'$ and using (ii) gives the isomorphism $P\iso P'$, hence (iv).

\end {proof}

\subsection{}\label{3.6}

Retain the above notation and take $V$ to be the finite dimensional $U(\mathfrak g)$ module we have fixed.

\begin {cor} $\mathscr F(V \otimes L_i)$ is the unique simple $A^\ltimes - U(\mathfrak g)$ bimodule quotient of $\mathscr F(V\otimes P_i)$.

\end {cor}

\begin {proof} Exactly the same proof as that of the first part of Proposition \ref {3.5}(iii) applies.  Alternatively one may apply Proposition \ref {2.6} to Lemma \ref {3.5}.
\end {proof}

\subsection{}\label{3.7}
 We note the following extension of the Bernstein-Gelfand \cite {BG} equivalence of categories.  It will not be used.

\begin {thm}  If $\lambda$ is regular, then $\mathscr F:\mathscr O^V_\Lambda \rightarrow \mathscr H^V_\lambda$ is an equivalence of categories with inverse functor $\mathscr F'$.
\end {thm}

\begin {proof}   Take $L \in \mathscr O^V_\Lambda$ simple. If $\lambda$ is regular, then $\mathscr FL(\mu)\neq 0$, for all $\mu \in \Lambda$ (see Lemma \ref {6.2} or \cite [8.4.1]{J2}).  Then by Proposition \ref {2.6} and Lemma \ref {3.3}, $\mathscr F L$ is simple and non-zero. Recalling \ref {3.2}, set $M:=\Ker (\mathscr F' \mathscr F L \twoheadrightarrow L)$. Since $\mathscr F$ is exact and covariant $\mathscr FM=\Ker(\mathscr F\mathscr F' \mathscr F L \twoheadrightarrow \mathscr FL)$. Then by Lemmas \ref {3.2} and  \ref {3.3}, $\mathscr FM=0$.  Yet $M \in \mathscr O^V_\Lambda$ so has finite length, whilst as we have seen $\mathscr FL\neq 0$ for every simple in this category.
 Hence $M=0$.  Then the conclusion obtains from Lemma \ref {3.1}(ii).
\end {proof}

\section{Further Consequences of Frobenius Reciprocity}

\subsection{}\label{4.1}

Let $N$ be a finite dimensional $B_\lambda$ module.  Then $A\otimes_BN$ is a left $A$ module and hence a left $A^\ltimes$ module.  We claim that it also admits a right $U(\mathfrak g)$ module structure obtained via right
multiplication on the first factor.  Indeed for all $x \in
\mathfrak g$ the map $(a\otimes bn,x)\mapsto (ax,bn)$ is well
defined since $(ab\otimes n,x)$ is mapped to $(abx\otimes
n)=(axb\otimes n)=(ax\otimes bn)$, by the invariance of $b \in B$.
Note that this right action is compatible with right action of $U(\mathfrak g)$ on $A$ introduced in \ref {1.1}.

Since $N$ is finite dimensional it follows from the above
that $A\otimes_BN = A_\lambda \otimes_B N $ has finite multiplicities for the diagonal action of $U(\mathfrak g)$.  Moreover under the right action of $Z(\mathfrak g)$, the ideal $\ann_{Z(\mathfrak g)}M(\lambda)$ annihilates $A\otimes_BN$, whilst the left action of $Z(\mathfrak g)$ is finite.  Thus $A\otimes_BN \in \mathscr H^V_\lambda$.

\subsection{}\label{4.2}

Recall that the standard form of Frobenius reciprocity asserts the following.  Let $R,R^\prime$ be rings, let $M$ be a $R'-R$ bimodule, $N$ a left $R$ module and $P$ a left $R'$ module. Then Frobenius reciprocity is the isomorphism of additive groups
$$F:\Hom_{R'}(M\otimes_R N,P) \iso \Hom_R(N,\Hom_{R'}(M,P)), \eqno {(*)}$$
given by $F_\psi(n)(m) = \psi(m\otimes n)$.

Let $N$ be a finite dimensional $B_\lambda$ module. Note that $A\otimes_BN$ identifies with $A_\lambda \otimes_BN$.  Again $A$ can be viewed as a left $R':=A^\ltimes \otimes U(\mathfrak g)^{op}$ module and a right $R:=B$ module as these two actions commute (because $B$ consists of $\mathfrak k$ invariants).

\begin {prop} For all $M' \in \mathscr O^V$, Frobenius reciprocity gives a vector space isomorphism
$$\Hom_{A^\ltimes-U(\mathfrak g)}(A \otimes_B N, F(M(\lambda),M')) \iso \Hom_B(N,N'), \eqno{(**)}$$
where $N'$ is the $B_\lambda$ module $\Hom_{U(\mathfrak g)}(M(\lambda),M')$.
\end {prop}

\begin {proof} Since $A_\lambda\otimes_BN \in \mathscr H^V_\lambda$, we can replace $F(M(\lambda),M'))$ in the left hand side by $\Hom(M(\lambda), M')$. \emph{Set} $N'':=\Hom_{A^\ltimes-U(\mathfrak g)}(A_\lambda,F(M(\lambda), M'))$. Then by $(*)$ the left hand factor maps isomorphically to $\Hom_B(N,N'')$.

Now $A_\lambda \in \mathscr H^V_\lambda$ and so by Lemma \ref {3.1}(ii), Frobenius reciprocity gives an isomorphism of $N''$ onto $\Hom_{A^\ltimes}(A_\lambda \otimes_{U(\mathfrak g)}M(\lambda), M')$.  In this we may replace $A_\lambda \otimes_{U(\mathfrak g)}M(\lambda)$, by $\End V \otimes_\mathbb C M(\lambda)$.  Then by \ref {2.4} $(**)$, Frobenius reciprocity gives an isomorphism of $N''$ onto $\Hom_{U(\mathfrak g)}(M(\lambda), \Hom_{\End V}(\End V,M'))$, which is just $N'$, as required.
\end {proof}

\section{Proof of Main Theorem}

\subsection{}\label{5.0}

We start this section with a few remarks which explain yet another reason why we took a rather asymmetric definition of the $U(\mathfrak g)$ bimodule structure of $A$ in \ref {1.1}.

Let $M,N,V'$ be $U(\mathfrak g)$ modules with $V'$ finite dimensional.  Given $\xi \in V'^*$ define a map $\xi:V'\otimes N \rightarrow N$, by evaluation on the first factor.  Then Frobenius reciprocity gives an isomorphism $F:\Hom(M,V'\otimes N)\iso \Hom(M\otimes V'^*,N)$ of vector spaces given by $F_\psi(m\otimes \xi)=\xi(\psi(m))$. Notice that since $V'$ is finite dimensional, we may interchange $V'$ and $V'^*$ in the above.  On the other hand although both sides are $U(\mathfrak g)$ bimodules, this is not an isomorphism of $U(\mathfrak g)$ bimodules.  Indeed let us write $\psi(m)=v'\otimes n$, with the convention that the right hand side may be a sum.  Then for all $x \in \mathfrak g$, one has
$$xF_\psi(m\otimes \xi)=\xi(v')xn, \quad F_\psi x(m\otimes \xi)= \xi(\psi(xm))+(x\xi)(v')n,$$
whilst
$$F_{x\psi}(m\otimes \xi)=\xi(v')xn+\xi(xv')n, \quad F_{\psi x}(m\otimes \xi)= \xi(\psi(xm)).$$

On the other hand $$[(xF_\psi-F_\psi x) - (F_{x\psi}-F_{\psi x})](m \otimes \xi)=-((x\xi)(v')+\xi(xv'))n=0.$$

Thus $F$ is an isomorphism of $U(\mathfrak k)$ modules, as is well-known.

Now replace $N$ by $V\otimes N$ in the above. Taking $M=N = M(\lambda), V'=V^*$, it follows that
$$F:F(M(\lambda), V\otimes (V^* \otimes M(\lambda))\iso F(V\otimes M(\lambda), V\otimes M(\lambda)),\eqno {(*)}$$
is an isomorphism of $U(\mathfrak k)$ modules (but not an isomorphism of $U(\mathfrak g)$ bimodules).

On the other hand our \textit{definition} of the $U(\mathfrak g)$ bimodule structure of $A$ means that it is the left hand side of $(*)$ which is isomorphic to $A_\lambda$ as a $U(\mathfrak g)$ bimodule.  In particular it is the left hand side of $(*)$ which is isomorphic to $A_\lambda$ as a left $U(\mathfrak g)$ module.

Taking $\mathfrak k$ invariants we conclude that

\begin {lemma}
$\Hom_{U(\mathfrak g)}(M(\lambda), V\otimes (V^* \otimes M(\lambda))$ is isomorphic to $B_\lambda$ as a left $B$ module.

\end {lemma}

\subsection{}\label{5.1}

Let $P$ be an indecomposable direct summand $V^*\otimes M(\lambda)$.  Since the latter is projective in $\mathscr O$, it follows that $P$ is an indecomposable projective in $\mathscr O$ and hence the projective cover $P(\mu)$ of a simple highest weight module $L(\mu)$.  Thus we have surjective $U(\mathfrak g)$ module maps $V^*\otimes M(\lambda)\twoheadrightarrow P(\mu) \twoheadrightarrow L(\mu)$.  Then by Frobenius reciprocity the $B_\lambda$ module $S:=\Hom_{U(\mathfrak g)}(M(\lambda),V\otimes L(\mu))$ is non-zero.  A fortiori $K:=\mathscr F (V\otimes L(\mu))=F(M(\lambda),V\otimes L(\mu))\neq 0$ and by Lemma \ref {3.3} is a simple $A^\ltimes - U_\lambda$  bimodule.  Furthermore by Corollary \ref {3.6}, it is the unique simple $A^\ltimes-U_\lambda$ quotient of $H:=F(M(\lambda),V\otimes P(\mu))$. The invariant summand $M:=\Hom_{U(\mathfrak g)}(M(\lambda),V\otimes P(\mu))$ is a $B_\lambda$ module and we obtain by restriction a $B_\lambda$ module surjection of $M$ onto $S$.

Consider $M$ as a $B_\lambda$ submodule of the $A^\ltimes -U_\lambda$ module $H$. We write $AM$ for the subspace of $H$ generated by the left action of $A$ or equivalently of $A^\ltimes$ on $M$.  It is a non-zero $A^\ltimes-U_\lambda$ bisubmodule of $H$, and just the image of $A\otimes_BM$ in $H$, under the multiplication map $\mu:a\otimes m \mapsto am$.  A similar result holds replacing the pair  $(H,M)$ by the pair $(K,S)$.

\begin {lemma} The multiplication map $\mu$ is an $A^\ltimes - U(\mathfrak g)$ bimodule isomorphism of $A_\lambda\otimes_{B_\lambda}M$ onto $H$.
\end {lemma}

\begin {proof}
 The non-zero $A^\ltimes-U_\lambda$ bimodule map  $\mu: A\otimes_BS \rightarrow K$ is surjective by the simplicity $K$.  Recalling the $B_\lambda$ module surjection $M\twoheadrightarrow S$ we obtain an $A$ module surjection
$A\otimes_BM \twoheadrightarrow A\otimes_BS$, since tensor product is right exact.  Thus we obtain a commutative diagram

$$
\begin{array}{ccc}
   A\otimes_BM & {\twoheadrightarrow}& A\otimes_BS\\
 \mu\downarrow && \twoheaddownarrow  \mu \\
 H&\stackrel{\psi}{\longrightarrow} & K
\end{array}
$$
forcing $\psi$ to be surjective. Thus $\im(A\otimes_B M \stackrel{\mu}{\rightarrow}H)$ surjects under $\psi$ to the unique simple quotient $K$ of $H$. Hence $A\otimes_B M \stackrel{\mu}{\twoheadrightarrow}H$.

The above conclusion holds for every direct summand $P_i$ of $V^*\otimes M(\lambda)$, that is to say (setting $M_i=\Hom_{U(\mathfrak g)}(M(\lambda),V\otimes P_i), H_i=F(M(\lambda),V\otimes P_i)$), one has
$$A_\lambda\otimes_{B_\lambda}M_i:=A_\lambda\otimes_{B_\lambda} \Hom_{U(\mathfrak g)}(M(\lambda),V\otimes P_i) \twoheadrightarrow F(M(\lambda),V\otimes P_i)=:H_i. \eqno{(*)}$$

 By Lemma \ref {5.0}, $\Hom_{U(\mathfrak g)}(M(\lambda), V\otimes (V^* \otimes M(\lambda))$ is just $B_\lambda$ as a left $B$ module.  Thus $A_\lambda=A_\lambda \otimes_{B_\lambda}B_\lambda = \oplus_i A\otimes_BM_i$, whilst as noted in \ref {5.0}, one has  $A_\lambda \iso F(M(\lambda), V\otimes (V^* \otimes M(\lambda))=\oplus_i\mathscr F(V\otimes P_i)=\oplus_iH_i$, as $U(\mathfrak k)$ modules.  Comparison of multiplicity of simple $\mathfrak k$ submodules (which are finite by \ref {3.4}) forces injectivity in $(*)$.
\end {proof}

\subsection{}\label{5.2}

We now turn to a proof of Theorem \ref {1.7}.

Let $N$ be a simple $B_\lambda$ module.  Then $A\otimes_BN \in \mathscr H^V_\lambda$.  Consequently by Corollary \ref {3.4}, there exists $M \in \mathscr O^V$ such that $A\otimes_BN \iso F(M(\lambda),M)$, as $A^\ltimes-U(\mathfrak g)$ bimodules.  By definition of the $U(\mathfrak g)$ bimodule structure of $A\otimes_BN$ given in \ref {4.1}, one has
$$(A\otimes_BN)^\mathfrak k=A^\mathfrak k\otimes_BN=B\otimes_BN=N. \eqno{(*)}$$

On the other hand $(A\otimes_BN)^\mathfrak k \iso (F(M(\lambda),M))^\mathfrak k = \Hom_{U(\mathfrak g)}(M(\lambda),M)$.   Since $A^\ltimes$ modules in $\mathscr O^V$ have finite length there exists a simple $A^\ltimes$ subquotient $K \in \mathscr O^V$ of $M$ such that $\Hom_{U(\mathfrak g)}(M(\lambda),K)$, is a non-zero $B$ subquotient of $N$, hence equal to $N$ by the simiplicity of the latter. (Since $M(\lambda)$ is projective in $\mathscr O$, there is just one such simple subquotient in a given composition series for $M$.) Finally by Proposition \ref {2.6}, we can write $K=V\otimes L(\mu)$ for some simple highest weight module $L(\mu) \in \mathscr O$.  This proves (i) of Theorem \ref {1.7}.

Assume $N:=\Hom_{U(\mathfrak g)}(M(\lambda),V \otimes L(\mu))$ non-zero and let $N'$ be a simple $B$ quotient. By (i) of Theorem \ref {1.7} we can write $N'=\Hom_{U(\mathfrak g)}(M(\lambda),V \otimes L')$, for some simple highest weight module $L' \in \mathscr O$. Then by Proposition \ref {4.2}, we conclude that
$$\Hom_{A^\ltimes-U(\mathfrak g)}(A \otimes_B N, F(M(\lambda),(V\otimes L' ))\neq 0.$$

Recall Lemmas \ref {5.1} and \ref {3.3} and the above notation. Since tensor product is right exact we obtain $A^\ltimes -U(\mathfrak g)$ modules surjections $F(M(\lambda),V\otimes P(\mu)) \iso A\otimes_B \Hom_{U(\mathfrak g)}(M(\lambda),V\otimes P(\mu))\twoheadrightarrow A\otimes_B N \twoheadrightarrow F(M(\lambda),V\otimes L')$.  On the other hand by Corollary \ref {3.6}, $F(M(\lambda),V\otimes L(\mu)$ is the unique simple $A^\ltimes -U(\mathfrak g)$ module quotient of $F(M(\lambda),V\otimes P(\mu))$.  This forces $F(M(\lambda),V\otimes L(\mu))\iso F(M(\lambda),V\otimes L')$.  Taking $\mathfrak k$ invariants of both sides gives $B$ module isomorphisms $\Hom_{U(\mathfrak g)}(M(\lambda),V\otimes L(\mu))\iso N'$, which proves (ii) of Theorem \ref {1.7}.

Now take $N=\Hom_{U(\mathfrak g)}(M(\lambda), V\otimes L(\mu)), N'=\Hom_{U(\mathfrak g)}(M(\lambda), V\otimes L(\mu')) $, in Proposition \ref {4.2}.  Assume that $N\iso N'$ as $B$ modules.  As in the proof of (ii) we obtain a $A^\ltimes - U(\mathfrak g)$ module surjection of $F(M(\lambda), V \otimes P(\mu))$ onto $F(M(\lambda), V\otimes L(\mu'))$ and then an isomorphism $F(M(\lambda), V \otimes L(\mu))\iso F(M(\lambda), V\otimes L(\mu'))$ of simple $A^\ltimes - U(\mathfrak g)$ bimodules. By Proposition \ref {2.6}, this gives an isomorphism $F(M(\lambda), L(\mu))\iso F(M(\lambda), L(\mu'))$ of simple $U(\mathfrak g)$ bimodules. By Proposition \ref {3.5}(iv), we conclude that $\mu=\mu'$, as required.

\subsection{}\label{5.3}

Let $\{P_i\}$ denote the set of pairwise non-isomorphic indecomposable direct summands of $V^* \otimes M(\lambda)$. Recall that $P_i$ is the projective cover of some simple $L_i \in \mathscr O$. Set $Q_i:=\Hom_{U(\mathfrak g)}(M(\lambda),V\otimes P_i)$. By Frobenius reciprocity (as we have already seen above) the $S_i:=\Hom_{U(\mathfrak g)}(M(\lambda),V\otimes L_i)$
 form the set of non-zero $B_\lambda$ modules.

\begin {lemma}  $Q_i$  is the projective cover of $S_i$.
\end {lemma}

\begin {proof}  Since $Q_i$ is a direct summand of $B_\lambda$ as a left $B_\lambda$ module, it is projective.  We remark that its projectivity may also be verified by the following argument.

Let $R \twoheadrightarrow S$ be a surjection of $B_\lambda$ modules.  Since tensor product is right exact we obtain a surjection $A\otimes_BR \twoheadrightarrow A\otimes_BS$ of $A^\ltimes - U_\lambda$ bimodules.  By Corollary \ref {3.4}, we may write $A\otimes_BR =F(M(\lambda),V\otimes M), A\otimes_BS=F(M(\lambda),V\otimes N)$. Taking invariants and using \ref {5.2}$(*)$ gives
$$R=\Hom_{U(\mathfrak g)}(M(\lambda),V\otimes M),\quad S=\Hom_{U(\mathfrak g)}(M(\lambda),V\otimes N).\eqno {(*)}$$
Now $A\otimes _B Q_i= F(M(\lambda),V\otimes P_i)$ is projective in $\mathscr H^V_\lambda$.  Thus we obtain a surjection $\Hom_{A^\ltimes - U_\lambda}(A\otimes _B Q_i,F(M(\lambda),V\otimes M))\twoheadrightarrow \Hom_{A^\ltimes - U_\lambda}(A\otimes _B Q_i,F(M(\lambda),V\otimes N))$.  This translates via Proposition \ref {4.2} and $(*)$ to a surjection $\Hom_B(Q_i,R)\twoheadrightarrow \Hom_B(Q_i,S)$, as required for projectivity.

Finally suppose that there is a non-zero map of $Q_i$ to the simple module $S_j$. Then by Proposition \ref {4.2}, there is non-zero $A^\ltimes - U_\lambda$ bimodule map of $A\otimes_B Q_i$ onto the simple module $F(M(\lambda), V \otimes L_j)$.  Yet by Lemma \ref {5.1}, the former is just $F(M(\lambda),V \otimes P_i)$, which by Proposition \ref {3.5}(iii) is the projective cover of $F(M(\lambda), V \otimes L_i)$. Moreover $i=j$, as required.
\end {proof}



\section{Non-Vanishing and Dimensions of the Simple $B_\lambda$ Modules}

Most of the statements in this section are well-known.   We repeat details for completeness.

Recall that $\lambda \in \mathfrak h^*$ is assumed dominant.

 Let $V(\nu):\nu \in P$ be the unique simple finite dimensional $U(\mathfrak g)$ module with extremal weight $\nu$, that is to say there exists $w \in W$ such that $w\nu \in P^+$ and $V(\nu)$ has highest weight $w\nu$.

 In the first two sections, let $V$ be an arbitrary finite dimensional $U(\mathfrak g)$ module.

\subsection{}\label{6.1}

\begin {lemma}  $\dim \Hom_{U(\mathfrak g)}(M(\lambda),V\otimes M(\mu))=\dim V_{\lambda-\mu}$.
\end {lemma}

\begin {proof}  As noted in say  \cite [8.1.7]{J2}), $V\otimes M(\mu)$ admits a Verma flag with $M(\nu)$ occurring $\dim V_{\nu - \mu}$ times. Since $M(\lambda)$ is projective in $\mathscr O$ one has $\dim \Hom_{U(\mathfrak g)}(M(\lambda),M(\nu))\leq 1$, with equality if and only if $\lambda=\nu$ (for example see \cite [Prop. 8.2.2(i)]{J2}).  Then through the above flag, again using projectivity, the assertion of the lemma obtains.
\end {proof}

\

\textbf{Remark.} As far as we know this observation (using projectivity) was first made in \cite {BG}.  An extension of this result which the above authors missed (and originally thought to be untrue) is given in \cite [Prop. 3.4]{GJ}, being partly derived from earlier work of the present author (see references in loc. cit.).

\subsection{}\label{6.2}

Define an order relation on $\mathfrak h^*$ by $\mu \geq \nu$ given $\mu-\nu \in P^+$. Recall \ref {1.5}, the translated action $(w,\lambda)\mapsto w.\lambda:=w(\lambda+\rho)-\rho$ of $W$ on $\mathfrak h^*$.  We call $\mu \in \lambda +P$ minimal if it is the unique minimal element in its $\stab_{W_\lambda.}\lambda$ orbit, that is to say minimal in the orbit given by the translated action $\{w \in W_\lambda|w.\lambda=\lambda\}$.  Let $( \ , \ )$ be the Cartan inner product on $\mathfrak h^*$.  Given $M \in \mathscr O$, let $[M]$ be its representative in the Grothendieck group.

 \begin {lemma}  $\mathscr FL(\mu)\neq 0 \Leftrightarrow \mu \in \lambda +P$ and is minimal.
 \end {lemma}

 \begin {proof}
 For $\Rightarrow$ suppose that $\mu$ is not minimal.
 Then there exists $s_\alpha \in \stab_{W_\lambda.}\lambda$, such that $\mu - s_\alpha.\mu$ is a positive (and necessarily integer) multiple of $\alpha$.  Then (see \cite [8.2.1]{J2}, for example) $M(s_\alpha.\mu)$ is a proper submodule of $M(\mu)$ and so by the projectivity of $M(\lambda)$, we obtain from $(*)$ that
 $ \dim \Hom_{U(\mathfrak g)}(M(\lambda),V\otimes M(\mu)/M(s_\alpha.\mu))=\dim V_{\lambda-\mu}-\dim V_{\lambda-s_\alpha.\mu} =\dim V_{\lambda-\mu}-\dim V_{s_\alpha(\lambda-\mu)}=0$. Since $V$ is arbitrary, $\mathscr F(M(\mu)/M(s_\alpha.\mu))=0$.  Since $\mathscr F$ is exact, this gives $\mathscr FL(\mu)=0$, as required.

 For $\Leftarrow$, set $\mathscr S_\mu:= \{\gamma \in \Delta_\lambda \cap \Delta^+|s_\gamma.\mu < \mu\}$. Take $\gamma \in \mathscr S_\mu$. By the minimality of $\mu$ we obtain $(\lambda +\rho,\gamma) >0$.  One checks that $(\lambda-s_\gamma.\mu,\lambda-s_\gamma.\mu) =(\lambda-\mu,\lambda-\mu)+2\gamma^\vee(\mu +\rho)(\lambda+\rho,\gamma)>(\lambda-\mu,\lambda-\mu)$. By \ref {6.1}$(*)$, this means that the ``minimal $\mathfrak k$-type" occurring in $F(M(\lambda),M(\mu))$, namely $V(\lambda-\mu)$ cannot occur in $F(M(\lambda),M(s_\gamma.\mu))$.

 Let $M(\mu)\supset M^1(\mu)\supset M^2(\mu)\supset\ldots,$ be the Jantzen filtration of $M(\mu)$ and recall that $L(\mu)=M(\mu)/M^1(\mu)$. The Jantzen sum formula (\cite [Satz 5.3] {Ja}, or \cite [4.4.17]{J2}) gives
$$\sum_{i\in \mathbb N^+}[M^i(\mu)]= \sum_{\gamma \in \mathscr S_\mu}[M(s_\gamma.\mu)].$$

From this formula and the projectivity of $M(\lambda)$, it follows from the previous observation that $V(\lambda-\mu)$ cannot occur in $F(M(\lambda),M^1(\mu))$.  Hence $F(M(\lambda),L(\mu))\neq 0$, as required.

 \end {proof}

\textbf{Remark.}  This result is contained in \cite [5.4]{BG}.  It is related to results of Zhelobenko as carefully worked out in \cite [Sect. I, 4.1, 4.7]{Du}.

 \subsection{}\label{6.3}

 \begin {cor}  Suppose $S(\mu)\neq 0$, then $V_{\lambda - \mu}\neq 0$ and $\mu$ is minimal.
 \end {cor}

 \

 \textbf{Remark.} Unlike Lemma \ref {6.2} it is not obvious if the converse holds.

\subsection{}\label{6.4}

By \ref {6.1}$(*)$ we obtain.

\begin {lemma}  For all $\lambda \in \mathfrak h^*$ one has $\dim B_\lambda = \dim (\End V)_0$.
\end {lemma}

\subsection{}\label{6.5}

Recall the discussion in the beginning of Section \ref {3.5} and let $[M(\lambda)\otimes V^*:P(\mu)]$ denote the number of direct summands of $M(\lambda)\otimes V^*$ isomorphic to $P(\mu)$.  Since $\dim \Hom_{U(\mathfrak g)}(P(\mu),L(\nu))\leq 1$ with equality if and only if $\nu =\mu$, we obtain the

\begin {lemma}  $\dim S(\mu)=[M(\lambda)\otimes V^*:P(\mu)]$.
\end {lemma}

\textbf{Remark.}   Combined with Lemma \ref {5.1}, recalling the definition of $Q_i$ given in \ref {5.3}, we obtain
$$B_\lambda \iso \oplus_i Q_i^{\dim S_i}. \eqno {(*)}$$

At first it seemed to the author that this was a surprising result; but it turns out to be a general fact (as I learnt from correspondence with J.T. Stafford and L. W. Small) for a finite dimensional algebra $A$ over (an algebraically closed field $\textbf{k}$).  Thus the lecture notes of R. Vale assert \cite [Thm. 4.14]{V} that every simple $A$ module has a projective cover and $A$ is a direct sum of its indecomposable projectives.  Earlier and in more generality this result is proved in \cite [Chap. I, Cor. 4.5]{ARS} for Artinian rings.  Neither reference describes their multiplicities (which seemed to me to be the surprising fact) but these can be recovered by an argument of Stafford which he kindly allowed me to reproduce below.

Let us recall some facts which can be found in \cite [Chap. I]{ARS} or in \cite [Chaps. 2,4]{H}. Let $A$ be an Artinian ring and $\{S_i\}_{i=1}^n$ its set of non-isomorphic simple modules.  Let $J(A)$ be the Jacobson radical of $A$ (just called the radical in \cite [Chap. I, Sect. 3]{ARS}).  Then $A/J(A)$ is a semisimple Artinian ring.  By definition of $J(A)$, every $S_i$ factors to a simple $A/J(A)$ module.   Moreover $D_i=\End_A S_i$ is a division algebra over which $S_i$ has some finite dimension $s_i$.  By Wedderburn's theorem $A/Ann_AS_i$ is a matrix algebra of size $s_i$ over $D_i$. Finally A/J(A) is a finite direct sum of these simple Artinian rings.   In particular as left $A$ modules
$$A/J(A) \iso \oplus_{i=1}^n S_i^{s_i}.$$

Of course if $A$ is a $\textbf{k}$ algebra with $\textbf{k}$ algebraically closed, then each $D_i$ is just $\textbf{k}$.

Let $P_i$ denote the projective cover of $S_i$.

 \begin {lemma} One has an isomorphism $A\iso \oplus_{i=1}^n P_i^{s_i}$ of left $A$ modules.

 \end {lemma}

 \begin {proof}
 Set $P=\oplus_{i=1}^n P_i^{s_i}$ which is projective and surjects to $A/J(A)$. Since $A$ is itself a projective left $A$ module, this factors to a map $\varphi:A \rightarrow P$.  If $\varphi$ is not surjective, then its image must be contained in a maximal submodule $M$.  By definition of $J(A)$ one has $M \supset J(A)$ and so the image of $M$ in $A/J(A)$ is the proper submodule $M/J(A)$, which is a contradiction.  (In brief, $\varphi$ is surjective by Nakayama's lemma.)  Finally since $P$ is projective, $\varphi$ admits a section $\sigma$ giving an isomorphism $A\iso \im \sigma \oplus\ker \varphi=P\oplus N$, with $N$ the left $A$ module $\ker \varphi$.  Yet if $N\neq 0$, it would have a simple factor which is not already a factor of $P$, contradicting the choice of $P$.  Hence the assertion of the lemma.

 \end {proof}

\subsection{}\label{6.6}

Recall \ref {1.5}.  Suppose $\Lambda$ is in general position. Then $M(\mu)=L(\mu)=P(\mu)$, for all $\mu \in \Lambda$.

\begin {lemma}  Suppose that $\Lambda$ is in general position.  Then $B_\lambda$ is isomorphic to $(\End V)_0$ as an algebra.
\end {lemma}

\begin {proof}  By Lemma \ref {6.1}, we obtain $\dim S(\mu)= \dim V_{\lambda - \mu}$.  By Theorem \ref {1.7} and the Jacobson density theorem, it follows that $B_\lambda$ contains the direct sum $\oplus_{\mu \in \Lambda} \End V_{\lambda - \mu}$.  Since the right hand side has dimension $\dim (\End V)_0$, the assertion follows from Lemma \ref {6.5}.
\end {proof}

\textbf{Remark 1.}  Since $\lambda$ is in general position there is a direct sum decomposition of $V\otimes M(\lambda)$ into pairwise non-isomorphic simple Verma modules.  From this the conclusion of the lemma obtains by direct computation.   Note the conclusion in Remark \ref {6.5} holds trivially in this case.

\textbf{Remark 2.}  Let $H$ be the algebra of functions on an affine supergroup whose zero graded part is reductive, that is to say $H$ is the Hopf dual of $U(\mathfrak g)$ where $\mathfrak g$ is a finite dimensional $\mathbb Z_2$ graded Lie algebra with $\mathfrak g_0$ reductive. Serganova \cite [Sect. 9]{S} obtained a result seemingly very close to that described in Remark \ref {6.5}.  Let $\{S_i\}_{i\in I}$ denote the set of simple modules and $P_i$ the projective cover of $S_i$ in the category of finite dimensional $U(\mathfrak g)$ modules.  Then
$$H=\oplus_{i \in I} P_i^{\dim S_i}.\eqno{(*)}$$
This is obtained by combining \cite [Thm. 9.2]{S} which makes an analogous statement for injective hulls with \cite [Lemma 9.5]{S} which shows that there is a distinguished one dimensional module $L$ such that the injective hull of $S\otimes L$ is the projective cover of $S$.   When $\mathfrak g= \mathfrak g_0$, $(*)$ is a familiar consequence of the algebraic Peter-Weyl theorem \cite [1.4.13]{J2}.

 The above result of Serganova is really quite different to that of the Lemma of Remark \ref {6.5}; the $S_i$ and $P_i$ are comodules for $H$ and the sum is infinite.  Relating it to the latter \ref {6.5}, if at all possible, will need further work.


\subsection{}\label{6.6.1}

At the other extreme suppose $\lambda+\rho =0$.  Then as is well-known $V(\nu)^*\otimes M(\lambda)$ is a direct sum of projective-injective modules $P(\Sigma)$ one for every $W$ orbit $\Sigma$ in the multi-set of weights of $V(\nu)^*$.  Moreover $P(\Sigma)$ is the projective cover of $L(\sigma-\rho)$, where $\sigma$ is the unique element in $\Sigma\cap -P^+$.  One may also remark that $[P(\Sigma)]=\sum_{w\in W/\stab_W\sigma}[M(w.(\sigma-\rho)]$, see \cite{CI}, for example. Thus by Lemma \ref {6.5}, we obtain the

\begin {cor} Suppose $\lambda= - \rho$.  Then $S(\sigma-\rho) \neq 0$ if and only if $-\sigma$ is a dominant weight of $V$ and moreover $\dim S(\sigma-\rho)= \dim V_{-\sigma}$.
\end {cor}

\textbf{Remark.}  In this strange case there are far fewer simple modules than in the generic case; but their dimensions are given by a similar formula, namely $\dim S(\nu)=\dim V_{-(\nu+\rho)}$, where $-(\nu+\rho)$ runs over the dominant weights of $V$.

\subsection{}\label{6.7}

We remark that Lemma \ref {6.5} gives a way to calculate the dimensions of the simple $B_\lambda$ modules using the values at $q=1$ of the Kazhdan-Lusztig polynomials.  Indeed the latter determine the multiplicity $[M(\mu):L(\nu)]$ of the simple module $L(\nu)$ in the Jordan-Holder series of the Verma module $M(\mu)$.  On the other hand every projective module $P \in \mathscr O$ admits a Verma flag, and we denote by $[P:M(\mu)]$, the number of occurrences of $M(\mu)$ in a Verma flag for $P$, which through formal character is independent of the choice of the flag.  Moreover by BGG reciprocity (\cite {BGG} and \cite {CI}) one has $[P(\nu):M(\mu)]=[M(\mu):L(\nu)]$, for all $\mu, \nu \in \mathfrak h^*$ and as is well-known that with respect to the order relation $<$, the matrix with these entries is triangular with ones on the diagonal, hence invertible.  On the other hand it is clear that
$$\sum_{\nu \in \Lambda}[M(\lambda)\otimes V^*:P(\nu)][P(\nu):M(\mu)]=[M(\lambda)\otimes V^*:M(\mu)]= \dim V_{\lambda-\mu},$$
from which the right hand side in Lemma \ref {6.5} may be computed.  Alternatively one may use the Kazhdan-Lusztig polynomials and BGG reciprocity to calculate the formal characters of $P(\nu)$ and combine this information with the easily calculated formal character of $M(\lambda)\otimes V^*$.

\section{A Further Equivalence of Categories}

In the following section we express the conclusions of our main theorem in a stronger and more formal manner.

\subsection{}\label{7.2}

Recall Remark \ref {3.4}.  Let $\mathscr G$ denote the covariant functor from the Harish-Chandra category $\mathscr H^V_\lambda$ of $A^\ltimes-U_\lambda$ modules to the category $\mathscr B_\lambda$ of finite dimensional $B_\lambda$ modules by taking $\mathfrak k$ invariants.  Since $\mathfrak k$ acts locally finitely and is semisimple (so its action is reductive) it follows that $\mathscr G$ is an exact functor.  Combining Theorem \ref {1.7} and Lemma \ref {5.3} we obtain a result which is very analogous to the combination of Lemma \ref {3.3} and Proposition \ref {3.5}.  The proof is also similar.

\begin {thm}  Take $H \in \mathscr H^V_\lambda$ simple and let $P$ denote its projective cover.

\

(i)   Either $\mathscr GH=0$, or $\mathscr GH$ is simple.  Moreover all the simple $B_\lambda$ modules are so obtained.

\

(ii)  If $\mathscr GH \neq 0$, then its projective cover is $\mathscr GP$.
\end {thm}

\subsection{}\label{7.3}

Let $\mathscr G'$ denote the covariant functor from the category $\mathscr B_\lambda$ of finite dimensional $B_\lambda$ modules to the category of $A^\ltimes-U_\lambda$ modules given by $\mathscr G'M:=A \otimes_BM$.  By virtue of Corollary \ref {3.4} we can write the conclusion of Proposition \ref {4.2} in the form

\begin {lemma} For all $H \in \mathscr H^V_\lambda, N \in \mathscr B_\lambda$ one has an isomorphism
$$Hom(N,\mathscr GH) \iso Hom(\mathscr G^\prime N,H),$$
where $Hom$ means homomorphism in the category.

\end {lemma}

\subsection{}\label{7.4}

The functor $\mathscr I:= \mathscr G \mathscr F \mathscr T$ is an exact covariant functor from $\mathscr O_\Lambda$ to $\mathscr B_\lambda$ given by

$$\mathscr IM=\Hom_{U(\mathfrak g)}(M(\lambda),V\otimes M).$$
By Proposition \ref {2.6}, Lemma \ref {3.3}. Proposition \ref {3.5} and Theorem \ref {7.2}, it takes a simple module to a simple module or zero and in the former case takes its projective cover to the projective cover of its image.

Set $\mathscr I':= \mathscr T' \mathscr F' \mathscr G'$.  It is a right exact covariant functor from $\mathscr B_\lambda$ to $\mathscr O_\Lambda$. Combining \ref {2.3}$(*)$, Lemma \ref {3.1}(ii) and Lemma \ref {7.3}, we obtain

\begin {lemma} For all $M \in \mathscr O_\Lambda, N \in \mathscr B_\lambda$ one has an isomorphism
$$Hom(N,\mathscr IM) \stackrel{\Psi}{\iso} Hom(\mathscr I^\prime N,M),$$
where $Hom$ means homomorphism in the category.

\end {lemma}

\subsection{}\label{7.5}

As in \ref {3.2}, we obtain for all $N \in \mathscr B_\lambda$, a map $\psi^-_N:N\rightarrow \mathscr I \mathscr I' N$ given by $\psi^-_N:=\Psi^{-1}(\Id_{\mathscr I'N})$  and for all $M \in \mathscr O_\Lambda$ a map $\psi^+_M: \mathscr I' \mathscr I M \rightarrow M$, given by $\psi^+_M:=\Psi(\Id_{\mathscr IM})$.

\begin {lemma}

(i) For all $P \in \mathscr O_\Lambda$ projective, $\psi^+_P$ is an isomorphism of $\mathscr O_\Lambda$ modules.

\

(ii) For all $N \in \mathscr B_\lambda$, $\psi^-_N$ is an isomorphism of $\mathscr B_\lambda$ modules.
\end {lemma}

\begin {proof}  Suppose $Q(\mu)\neq 0$.  It is immediate that $\mathscr IP(\mu)=Q(\mu)$.  On the other hand $\mathscr I'Q(\mu)=P(\mu)$, obtains from Lemma \ref {5.1}.  Hence (i) and (ii) hold for projective modules in $\mathscr B_\lambda$.  Then by taking the first two steps of a projective resolution of $N \in \mathscr B_\lambda$, a standard argument (as in \cite [Prop. 8.4.4(i)]{J2}) gives the required assertion.
\end {proof}

\subsection{}\label{7.6}

The natural analogues of the Verma modules in the category $\mathscr B_\lambda$ of finite dimensional $B_\lambda$ modules are the non-vanishing $N(\mu):=\mathscr I M(\mu)$. For example just as the formal character of a Verma module is easy to compute, it is immediate from Lemma \ref {6.1}, that
$$\dim N(\mu)= \dim V_{\lambda-\mu}.$$

Recall that $S(\nu) =\mathscr I L(\nu)$.

\begin {cor}  Assume that $N(\mu) \neq 0$ and $S(\nu)\neq 0$.  There is a canonical vector space isomorphism
$$\Ext^*_\mathscr O(M(\mu),L(\nu))\iso \Ext^*_{\mathscr B_\lambda}(N(\mu),S(\nu)). \eqno {(*)} $$
\end {cor}

\begin {proof}  Indeed if $P^*\rightarrow M(\mu)$ is a projective resolution of
$M(\mu)$, then $\mathscr IP^* \rightarrow \mathscr IM(\mu)$ is a projective resolution of $N(\mu)=\mathscr IM(\mu)$.
Thus by Lemmas \ref {7.4} and \ref {7.5}(i), we obtain $\Ext^*_{\mathscr B_\lambda}(N(\mu),S(\nu)):= \Hom_{\mathscr B_\lambda}(\mathscr IP^*,\mathscr I L(\nu))\iso \Hom_{\mathscr O_\Lambda}(\mathscr I' \mathscr IP^*,L(\nu))\iso \Hom_{\mathscr O_\Lambda}(P^*,L(\nu))=:\Ext^*_{\mathscr O_\Lambda}(M(\mu),L(\nu))$, as required.

\end {proof}

\textbf{Remark.}  As is well-know the information on the left hand side is given by the Kazhdan-Lusztig polynomials and conversely determines these polynomials.  Thus the Corollary gives a second interpretation of these polynomials via the representations of relative Yangians, which we recall are finite dimensional algebras.
\subsection{}\label{7.7}

Of course $\mathscr I$ is not an equivalence of categories because it can happen that $\mathscr IL(\nu)=0$.

We may obtain an equivalence of categories by resorting to direct sums.

Let us replace $V$ by possibly infinite direct sum $\oplus_{i\in I} V_i$ of finite dimensional
$U(\mathfrak g)$ modules and $\End V$ by  $\oplus_{i\in I} \End V_i$ throughout the above. It is practically immediate that the previous theory goes through without significant change. For example the simple $B_\lambda$ modules take the form $\Hom_{U(\mathfrak g)}(M(\lambda),V_i \otimes L(\nu))$ and those modules which are non-zero, are pairwise non-isomorphic.


Now let $V_i:i \in I$ be the set of all finite dimensional simple $U(\mathfrak g)$ modules, up to isomorphism.   Suppose that $F(M(\lambda),L(\nu))\neq 0$.  This means that there is a simple finite dimensional $U(\mathfrak g)$ module $V$ such that $0 \neq \Hom_{U(\mathfrak g)}(V,\Hom (M(\lambda),L(\nu))$.  Yet by a well-known version of Frobenius reciprocity in the context of enveloping algebras (or more generally algebras with a coproduct) the latter is isomorphic to $\Hom_{U(\mathfrak g)}(V\otimes M(\lambda),L(\nu))$, where the coproduct is used to define the action on $\Hom$ and on $\otimes$.   Similarly (but now also using that the coproduct is cocommutative) the latter is isomorphic to $\Hom_{U(\mathfrak g)}(M(\lambda),\Hom(V,L(\nu))$, which since $V$ is finite dimensional is in turn isomorphic to $\Hom_{U(\mathfrak g)} (M(\lambda),V^* \otimes L(\nu))$.  Finally since $V$ is finite dimensional we may replave $V$ by $V^*$ in the above to conclude that $\Hom_{U(\mathfrak g)} (M(\lambda),V \otimes L(\nu))\neq 0$ for some finite dimensional $U(\mathfrak g)$ module $V$. Hence we obtain

\begin {thm}  Under the above replacement, $\mathscr I : \mathscr O_\Lambda \rightarrow \mathscr B_\lambda$ is an equivalence of categories with inverse functor $\mathscr I'$.
\end {thm}


\begin{thebibliography}{m}

\bibitem {ARS}  M. Auslander, I. Reiten and S. O. Smal$\phi$, Representation theory of Artin algebras. Cambridge Studies in Advanced Mathematics, 36. Cambridge University Press, Cambridge, 1995.

\bibitem {BG} J. N. Bernstein and S. I.  Gelfand, Tensor products of finite- and infinite-dimensional representations of semisimple Lie algebras.  Compositio Math.  41  (1980), no. 2, 245–-285.

\bibitem {BGG} I. N. Bernstein, I.M.  Gelfand and S. I. Gelfand,  Structure of representations that are generated by vectors of highest weight. (Russian) Funckcional. Anal. i Prilozen.  5  (1971), no. 1, 1–-9.


\bibitem {CI} D. H. Collingwood and R. S.  Irving, A decomposition theorem for certain self-dual modules in the category O.  Duke Math. J.  58  (1989),  no. 1, 89–-102.

\bibitem {D} J. Dixmier,  Alg\`ebres enveloppantes. (French) [Enveloping algebras] Reprint of the 1974 original. Les Grands Classiques Gauthier-Villars. [Gauthier-Villars Great Classics] \'Editions Jacques Gabay, Paris, 1996.

\bibitem {Du} M. Duflo Repr\'esentations irr\'eductibles des groupes semi-simples complexes. (French)  Analyse harmonique sur les groupes de Lie (S\'em., Nancy-Strasbourg, 1973–75),  pp. 26–-88. Lecture Notes in Math., Vol. 497, Springer, Berlin,  1975.

\bibitem {GJ} O. Gabber and A. Joseph, On the Bernstein-Gelfand-Gelfand resolution and the Duflo sum formula.  Compositio Math.  43  (1981), no. 1, 107–-131.

\bibitem {H}  I. N. Herstein,  Noncommutative rings. Reprint of the 1968 original. With an afterword by Lance W. Small. Carus Mathematical Monographs, 15. Mathematical Association of America, Washington, DC, 1994.




\bibitem {Ja} J.-C. Jantzen, Moduln mit einem h\"ochsten Gewicht. (German) [Modules with a highest weight] Lecture Notes in Mathematics, 750. Springer, Berlin, 1979.

\bibitem {J0} A. Joseph, A surjectivity theorem for rigid highest weight modules.  Invent. Math.  92  (1988),  no. 3, 567–-596.

\bibitem{J2}   A. Joseph, Quantum Groups and their Primitive Ideals,
Springer-Verlag, 1995.

\bibitem {KN} S. Khoroshkin and M. Nazarov, Mickelsson algebras and representations of Yangians, TAMS, 364 (2012), no. 3, 1293--1367.

\bibitem{KNV} S. Khoroshkin, M. Nazarov and E. Vinberg, A generalized Harish-Chandra isomorphism,  Adv. Math.  226  (2011),  no. 2, 1168–-1180.

\bibitem{Ki1}  A. A. Kirillov, Introduction to family algebras. Mosc. Math. J. 1 (2001), no. 1, 49–-63.

\bibitem{Ki2}  A. A. Kirillov,  Family algebras and generalized exponents for polyvector representations of simple Lie algebras of type Bn. (Russian) Funktsional. Anal. i Prilozhen. 42 (2008), no. 4, 72--82, 112; translation in Funct. Anal. Appl. 42 (2008), no. 4, 308–-316

\bibitem {M}  A. Molev, Yangians and classical Lie algebras. Mathematical Surveys and Monographs, 143. American Mathematical Society, Providence, RI, 2007.

\bibitem {R}  N. Rozhkovskaya, Commutativity of quantum family algebras. Lett. Math. Phys. 63 (2003), no. 2, 87–-103.

\bibitem {S} V. Serganova,  Quasireductive supergroups, New developments in Lie theory and its Appli
cations, 141–-159, Contemp. Math., Amer. Math. Soc., Providence, RI, 2011.

\bibitem {V} R. Vale, 7350: Topics in finite dimensional algebras, Lecture Notes, Cornell, 2009.






\end{thebibliography}
\end{document}